\documentclass[reqno,11pt]{amsart}
\usepackage{amssymb, amsmath}
\usepackage{graphicx}
\usepackage{cite}

\usepackage{subfig, tikz}
\usetikzlibrary{decorations.pathmorphing}
\usepackage{graphics}

\newcommand{\koniec}{\begin{flushright}  $\Box $ \end{flushright}}

\newcommand{\td}{\text{d}}
\newcommand{\Rho}{\mathrm{P}}
\newcommand{\R}{\mathbb{R}}
\newcommand{\Z}{\mathbb{Z}}
\newcommand{\RP}{\mathbb{RP}}

\usepackage{bm}
\usepackage{enumerate}
\theoremstyle{plain}
\newtheorem{theorem}{Theorem}[section]
\newtheorem{lemma}[theorem]{Lemma}
\newtheorem{prop}[theorem]{Proposition}
\newtheorem{cor}[theorem]{Corollary}

\theoremstyle{definition}
\newtheorem{remark}[theorem]{Remark}
\newtheorem{definition}[theorem]{Definition}

\def\be{\begin{equation}}
\def\ee{\end{equation}}
\def\bea{\begin{eqnarray}}
\def\eea{\end{eqnarray}}

\def\xb{X^\flat}
\def\z{\zeta}
\def\zb{\bar{\zeta}}

\def\p{\partial}

\DeclareMathOperator{\arcsinh}{arcsinh}

\newcounter{mnotecount}[section]

\renewcommand{\themnotecount}{\thesection.\arabic{mnotecount}}

\newcommand{\mnote}[1]
{\protect{\stepcounter{mnotecount}}$^{\mbox{\footnotesize
$
\bullet$\themnotecount}}$ \marginpar{
\raggedright\tiny\em
$\!\!\!\!\!\!\,\bullet$\themnotecount: #1} }

\def\theequation{\thesection.\arabic{equation}}
\numberwithin{equation}{section}

\begin{document}
\title{New quasi-Einstein metrics on a two-sphere}

\author{Alex Colling}
\address{Department of Applied Mathematics and Theoretical Physics\\ 
University of Cambridge\\ Wilberforce Road, Cambridge CB3 0WA, UK.}
\email{aec200@cam.ac.uk}
\author{Maciej Dunajski}
\address{Department of Applied Mathematics and Theoretical Physics\\ 
University of Cambridge\\ Wilberforce Road, Cambridge CB3 0WA, UK.}
\email{m.dunajski@damtp.cam.ac.uk}

\author{Hari Kunduri}
\address{Department of Mathematics and Statistics and Department of Physics and Astronomy\\
McMaster University\\ Hamilton  Ontario L8S 4M1 Canada}
\email{kundurih@mcmaster.ca }

\author{James Lucietti}
\address{
School of Mathematics and Maxwell Institute for Mathematical Sciences\\ University of Edinburgh\\
King’s Buildings, Edinburgh, EH9 3JZ, UK.}
\email{j.lucietti@ed.ac.uk}


\begin{abstract}
We construct 
all axi-symmetric non-gradient $m$-quasi-Einstein structures on a two-sphere.
This includes the spatial cross-section of the extreme Kerr black hole horizon corresponding to $m=2$, as well as a family of new regular metrics with $m\neq 2$ given in terms of hypergeometric functions. 
We also show that in the case $m=-1$ with vanishing cosmological constant the only  orientable compact
solution in dimension two is the flat torus, which proves that there are no compact surfaces with a metrisable
affine connection with skew Ricci tensor. 
\end{abstract}

\maketitle

\section{Introduction}
A Riemannian manifold $(M, g)$ of dimension $n$ is said to be $m$-quasi-Einstein if there exists a vector field $X \in \mathfrak{X}(M)$ and constant $\lambda$ such that
\be
\label{QEE}
\mbox{Ric}(g)=\frac{1}{m}X^\flat\otimes X^\flat
-\frac{1}{2}{\mathcal L}_{X} g+\lambda g
\ee
where ${\mathcal L}_X$ is the Lie derivative,
the one-form $X^\flat$ is $g$-dual to $X$ with respect to the metric $g$, and
$m$ is a non-zero constant.  Two specific cases of interest are $m=2$,  where $(\ref{QEE})$ is the near horizon equation arising on spatial cross sections of extremal black hole horizons in General Relativity with cosmological constant $\lambda$ \cite{LP1, LP2, KL9, KL13}, and $m=1-n, \lambda=0$, which characterises
Levi-Civita connections  projectively equivalent
to affine connections with skew-symmetric Ricci tensor \cite{nur}.

 For generic $m$ a number of general results which constrain the topology and geometry of quasi-Einstein manifolds have been obtained in the case  the vector field $X$ is a gradient \cite{case, gradientref1, gradientref2} and more generally in the case  $X^\flat$ is closed \cite{Bahuaud:2022iao}  (see also
 \cite{CRT} if $m=2$).  In particular all
solutions to (\ref{QEE}) on a two-sphere $S^2$ are trivial (that is, $X$ vanishes identically) in this case. 
If $X^\flat$ is not closed and 
$m=2$, then the most general solution to
(\ref{QEE}) on $S^2$ belongs to the explicit two-parameter family of solutions arising from the extreme Kerr black hole horizon, which are parametrised by the mass and the cosmological constant~ \cite{DL23}. These solutions are all axi-symmetric.
 
 In this paper we ask the simple question: are there any non-trivial examples of  $m$-quasi-Einstein manifolds on $M=S^2$ for $m\neq 2$?   To the best of our knowledge no such example has been written down in the literature. Our main result is the following.
\begin{theorem}
\label{main_theo_md}
Let $(g, X)$ be a solution to the $m$-quasi-Einstein equation (\ref{QEE}) on a two-dimensional connected surface $M$, with $\mathrm{d} X^\flat$ not identically zero, and a $U(1)$ isometric action.
\begin{itemize}
\item
Locally there exist coordinates $(x, \phi)$ on $M$, and a function $B=B(x)$ such that
the quasi-Einstein structure is of the form
\be
\label{B_met}
g=B(x)^{-1}\mathrm{d}x^2+B(x) \mathrm{d}\phi^2, \quad 
\xb=\frac{-m}{x^2+1}\Big(x\mathrm{d} x-B(x) \mathrm{d}\phi\Big),
\ee
with
\be
\label{hyper_geom}
B(x)=
\begin{cases}
b x(x^2+1)^{-m/2}+c (x^2+1)^{-m/2} F(x)
-\frac{\lambda(x^2+1)}{m+1} & \text{for $m\neq -1$}\\
     x\Big(b-\lambda\arcsinh(x)\Big)\sqrt{x^2+1}+c(x^2+1)        & \text{for $m=-1$}
		 \end{cases}
\ee
where $b, c$ are constants and
\be
\label{hg_function}
F(x)\equiv _2F_1\Big(-\frac{1}{2}, -\frac{m}{2}, \frac{1}{2}, -x^2\Big)
\ee
is the hypergeometric function.
\item If $b=0$ and
\begin{center}
    \begin{tabular}{c|c|c|c}
   & $\lambda=0$ & $\lambda>0$ &$\lambda<0$\\
   \hline
 $m>0$ & $c>0$ & $c>\frac{\lambda}{m+1}$ &$c>\frac{|\lambda|}{m+1} c_0$\\  
 $m\in(-1,0)$ & - & $c>\frac{\lambda}{m+1}$ &-\\  
 $m=-1$ & - &  $c>0$ & -\\
 $m<-1$  & - & $c\in\Big(\frac{\lambda}{m+1},0\Big)$ &-
    \end{tabular}
\end{center}    
where
\[
c_0=\min_{x>x_0} \frac{(x^2+1)^{\frac{m}{2}+1}}{|F(x)|}
\]
and $x_0$ is the unique positive zero of $F$,
then (\ref{B_met}) smoothly extends to $S^2$.
\item Conversely all solutions to (\ref{QEE}) on $S^2$ 
with a $U(1)$ isometric action arise  from (\ref{B_met}) with
(\ref{hyper_geom}) and $b=0$, together with the restrictions on $c$ given in the above table.
\end{itemize}
\end{theorem}
Therefore $b=0$ together with the restriction on $c$ given in the second part of the Theorem are 
necessary and sufficient conditions for
the structure (\ref{B_met}) to extend to $S^2$. 
In the special case $m=2$, where the formula (\ref{hg_function}) reduces to $F(x)=1-x^2$,
a much stronger result has been established \cite{DL23}: the existence of a $U(1)$ isometry
preserving $X$ follows from the global assumption that $M=S^2$. This rigidity result has recently been extended to $m>2$ and $m\leq 0$ in \cite{alex}. In view of these results the additional assumption of
the existence of an $U(1)$ isometric action on $S^2$ is only needed if $m \in (0, 2)$.

To arrive at the local normal form (\ref{B_met}),
(\ref{hyper_geom}) we shall 
use a reformulation (Proposition \ref{prop1} in \S\ref{sectionkahler}) of (\ref{QEE}) as a fourth order non-linear PDE for the K\"ahler potential for the metric $g$. 
If $g$ admits a Killing vector, then either $[K, X]=0$, or $g$ is a metric of constant curvature (Proposition
\ref{prop_inh} in \S\ref{sectionlocal}). In this axi-symmetric case the PDE for the K\"ahler potential 
reduces to an ODE solvable in terms of the hypergeometric functions (Theorem \ref{theorem_normal_K} in 
\S\ref{sectionlocal}). To ensure that the solutions  (\ref{B_met}), (\ref{hyper_geom}) extend to smooth metrics on $S^2$, the function $B$ in (\ref{hyper_geom})
and its derivatives must satisfy certain regularity conditions at the end points (Lemma \ref{lem:gsmooth} in \S\ref{sectionregular}). Imposing these
conditions under the additional assumption that $B$ is an even function leads to the proof of the second part of Theorem \ref{main_theo_md}. The last part of the Theorem asserts that the even parity of $B$ is necessary as well as sufficient (Lemma \ref{parity_lemma} in \S\ref{sectionregular}).
In particular we deduce that for all combinations of $(m, \lambda)$ allowed by the Gauss-Bonnet theorem (Proposition \ref{prop2})  there exists an open set $U\subset \R$ such that the quasi-Einstein structure corresponding to (\ref{hyper_geom}) extends to $S^2$ iff $c\in U$.

While the paper focuses on the genus zero case, some non-existence results can be established for other topologies.  This has been done in \cite{Lew_genus} if $m=2$, and more recently \cite{alex} for general $m$. In \S\ref{section_proj} we shall use an old result of Milnor to establish the following.
\begin{theorem}
\label{mminus1}
Let $(g, X)$ satisfy the quasi-Einstein equation (\ref{QEE}) with $\lambda=0$ and $m=-1$ on a compact
orientable surface $M$. Then $X=0$ and $(M, g)$ is the flat torus.  
\end{theorem}
In the Appendix we shall consider the general dimension $n$ and show that, if $m=2$ and $M$ is 
compact, then the Killing vector of $g$ proven to exist in \cite{DL23} must preserve $X$. This strengthens the results of \cite{DL23}, where this {\em inheritance} property has been established
only for $n=2$, or for general $n$ but $\lambda\leq 0$. In the rest of the paper with the exception of Lemma \ref{RAlem}, some considerations in  \S\ref{section_proj}, and Theorem \ref{theoapp}  in the Appendix, we shall assume that $n=\dim{M}=2$.

\subsection*{Acknowledgements} MD is grateful to Jan Derezi\'nski for helpful discussions. JL acknowledges support by the Leverhulme Research Project Grant RPG-2019-355.   HK acknowledges the support of the NSERC Grant RGPIN-2018-04887. AC acknowledges the support of the Cambridge International Scholarship.
The work of AC and MD was partially supported by the Simons Foundation grant (award no. SFI-MPS-T-Institutes-00010825) and from State Treasury funds as part of a task commissioned by the Minister of Science and Higher Education under the project {\em Organization of the Simons Semesters at the Banach Center - New Energies in 2026-2028} (agreement no. MNiSW/2025/DAP/491).

\section{General properties of the quasi-Einstein equations on surfaces}
In this section we shall establish some  properties of (\ref{QEE}) needed in the rest of the paper. 

\subsection{Topology}
If $M$ is compact
and without boundary, then its topology is constrained by the sign of $m$ and $\lambda$.
\begin{prop}
\label{prop2}
Let  ${\tt g}_M$ be the genus of a two-dimensional closed orientable surface  $M$, and let $(M, g, X)$ satisfy (\ref{QEE}). 
\begin{itemize}
\item If $m>0$ and $\lambda>0$, then ${\tt g}_M=0$.
\item If $m>0$ and $\lambda=0$, then ${\tt g}_M\leq 1$ with an equality iff $(M, g)$ is the flat torus.
\item If $m<0$ and $\lambda<0$, then ${\tt g}_M>1$.
\item If $m<0$ and $\lambda=0$, then ${\tt g}_M\geq 1$ with an equality iff $(M, g)$ is the flat torus.
\end{itemize}
\end{prop}
\begin{proof} 
Taking the trace of (\ref{QEE}), and integrating the resulting equation over $M$, dropping the divergence term, and using the Gauss-Bonnet formula 
gives
\be
\label{gb_constraint}
\int_M\Big(\frac{|X|^2}{m}+2\lambda\Big)\mbox{vol}_M=8\pi(1-{\tt g}_M)  \; ,
\ee where $\mbox{vol}_M$ is the  volume form of $(M,g)$.
The statement of the proposition now follows immediately. In the genus one case with $\lambda=0$ we get $|X|^2=0$, which in the Riemannian signature
implies that $X=0$. Substituting this back into (\ref{QEE}) gives that $g$ is flat.
\end{proof}
\subsection{Prolongation}
The quasi-Einstein equations (\ref{QEE}) do not form a closed system, as only the symmetrised part of
the covariant derivative of $X^\flat$
 is specified. It is however a system of finite type which means
that adding additional unknowns, and cross-differentiating eventually determines all derivatives (see e.g.
Appendix C in \cite{D}). This process known as the prolongation has been carried over for (\ref{QEE})
in \cite{nur}, and it turns out that only one differentiation is needed. 

In the proof of Proposition \ref{prop_inh} we will need to go a step further, and explore
some of the constraints coming from the prolongation procedure applied to (\ref{QEE}) on a surface. We close the system by introducing a function $\Omega$ by d$X^\flat = \Omega\mbox{vol}_M$ to  obtain
\begin{equation}\label{pqe1}
\nabla X^\flat = \frac 1m X^\flat \otimes X^\flat + (\lambda - \tfrac 12 R)g + \frac 12\Omega\mbox{vol}_M. 
\end{equation}
Applying $\nabla$ to (\ref{pqe1}) and commuting the derivatives of $X$ leads to
\begin{equation} \label{pqe2}
\text{d} \Omega = \frac 3m \Omega X^\flat + \star\text{d} R + \frac{1}{m}\left(2\lambda - \left(m + 1\right)R\right)\star\!X^\flat \; ,
\end{equation}
where $R$ is the scalar curvature and $\star$ the Hodge star operator of $g$.
The system is now closed after one prolongation.
The first constraint comes from setting the skew part of $\nabla^2 \Omega$ to zero. It gives
\begin{equation} \label{step1}
0 = -\Delta R + \left(1 + \frac 4m\right)\langle X,\text{d} R\rangle + \frac{3}{m}\Omega^2 + \frac{1}{m^2}(2\lambda - (m+1)R)(2\vert X \vert^2 -2\lambda m + mR) \; ,
\end{equation}
where $\Delta$ is the Laplace-Beltrami operator.
Differentiating this and removing all derivatives of $X$ and $\Omega$ using (\ref{pqe1}) and (\ref{pqe2}) yields two more constraints from vanishing of two components of the following one--form
\begin{align}
    0 = &-\text{d}(\Delta R) + \frac 1m\left(1 + \frac 4m\right)\langle X, \text{d} R\rangle X^\flat + \left(1 + \frac 4m\right)\text{Hess}(R)(X) + \left(\frac 12 + \frac{8}{m}\right)\Omega\star\!\text{d} R \nonumber \\
 + &\left[- \frac 2m\left(1 + \frac 1m\right)\vert X \vert^2 - \left(\frac 52 + \frac 4m\right)R + \lambda\left(3 + \frac 8m\right)\right] \text{d} R + \frac{18}{m^2}\Omega^2X^\flat \nonumber \\
    &+ \frac{4}{m^2} \left(2\lambda - (m+1)R\right)\left(\left(\lambda - \tfrac 12R\right)X^\flat + 2\Omega\star\! X^\flat + \frac 1m\vert X \vert^2X^\flat\right). \label{step2}
\end{align}
\begin{remark}
\label{rem_const_R} In the constant curvature case (with $X$ non-identically zero) equations (\ref{pqe2}) and (\ref{step2}) imply $\Omega = 0$ and $2\lambda = (m+1)R$ if $m\neq -1$ (if $m=-1$, then $R$ can be arbitrary
but $\lambda$ must vanish).
\end{remark}

\subsection{Analyticity}
We will need to refer to the real analyticity of (\ref{QEE}) in our proof
of Proposition \ref{prop_inh} when $n=2$. The proof of the following result is however valid for any $n$. It follows the considerations by Pedersen and Tod for the Einstein-Weyl equation \cite{PT}.
\begin{lemma}
\label{RAlem}
Let $(M, g, X)$ be a quasi-Einstein manifold of dimension $n$. In harmonic coordinates the components of $X$ and $g$ are real analytic.    
\end{lemma}
\begin{proof}
The contracted Bianchi identity applied to the quasi-Einstein equations (\ref{QEE}) implies (see Lemma 2.1 in \cite{BGKW})
\begin{equation} \label{qe2}
\Delta X^\flat = -\lambda X^\flat + \left(\frac 2m + \frac 12\right)\nabla_X X^\flat - \left(\frac 1m - \frac 14\right)\nabla(\vert X\vert^2) - \frac 1m\vert X \vert^2X^\flat + \frac 2m(\text{div }X)X^\flat.
\end{equation}
The only terms in (\ref{QEE}) and (\ref{qe2}) that contain second derivatives are Ric$(g)$ and $\Delta X^\flat$. In a harmonic coordinate chart $(U, (x^i))$ on $M$ we can write
\begin{equation}
    \text{Ric}(g)_{ij} = -\frac{1}{2} g^{kl}\frac{\partial^2 g_{ij}}{\partial x^k\partial x^l} + \dots \hspace{.4cm} \text{and}\hspace{.4cm} \Delta X_i^\flat = g^{kl}\frac{\partial^2 X_i^\flat}{\partial x^k\partial x^l} + \dots,
\end{equation}
where the dots denote terms involving at most first derivatives of components of $g$ and $X^\flat$. Note that the coordinate expression for the Laplace-Beltrami operator acting on one-forms does contain second derivatives of the metric, but these can be eliminated using the harmonic coordinate condition $g^ {ij}\Gamma^k_{ij} = 0$ and the expression for $g^{kl}\partial_k\partial_l g_{ij}$ coming from (\ref{QEE}).  Since the metric is positive definite, it follows that in harmonic coordinates (\ref{QEE}) and (\ref{qe2}) constitute an elliptic system of $\frac 12n(n+3)$ differential equations in $\frac 12n(n+3)$ unknowns $g_{ij}$ and $X_k^\flat$. Following the arguments of Deturck and Kazdan \cite{DK} we deduce that the components of $g$ and $X$ are real analytic in these coordinates.
\end{proof}

\subsection{K\"ahler potential}
\label{sectionkahler}
In this section we shall reformulate (\ref{QEE}) as a scalar PDE for a K\"ahler potential. 
In \cite{DL23} this formulation was used in the particular case when $m=2$. We refer the reader
to \cite{DL23} for the general discussion of K\"ahler potentials and their global properties on $S^2$.
\begin{prop}
\label{prop1}
Locally there exist complex coordinates $(\z, \zb)$ and a K\"ahler potential $f$ on $M$ such that
\be
\label{gandV}
g=4f_{\zeta\bar{\zeta}} \; \td\zeta \td\bar{\zeta}, \quad 
\xb=-{mf_{\zeta\bar{\zeta}}}
(\td\zeta/f_{\bar{\zeta}}+\td\bar{\zeta}/f_{{\zeta}}),
\ee
and the system (\ref{QEE}) reduces to a single 4th order PDE for the K\"ahler potential
$f$
\begin{eqnarray}
\label{pde}
&&\frac{2}{m}(f_{\z}f_{\zb})^2(f_{\z\z\zb\zb}f_{\z\zb}-f_{\z\z\zb}f_{\z\zb\zb})
+\frac{4\lambda}{m}(f_{\z\zb})^3(f_{\z})^2(f_{\zb})^2
\\
&&-(f_{\z\zb})^3 (f_{\zb\zb}(f_{\z})^2 +f_{\z\z}(f_{\zb})^2)+
(f_{\z\zb})^2(f_{\z}(f_{\zb})^2f_{\z\z\zb}+f_{\zb}(f_{\z})^2f_{\z\zb\zb})
+2 (f_{\z\zb})^4f_{\z}f_{\zb}=0.\nonumber
\end{eqnarray}
\end{prop}
\begin{proof} 
Choose holomorphic coordinates and an associated K\"ahler potential so $g$ is of the form (\ref{gandV}). In these coordinates write
\[
\xb=A\td\zeta+\overline{A}\td\zb,
\]
where $A$ is a smooth complex valued function.
The $(\zeta\zeta)$ and  $(\zb\zb)$
components of (\ref{QEE}) yield
\[
\p_\zeta A-\frac{1}{m}A^2-\frac{f_{\zeta\zeta\zb}}{f_{\zeta\zb}}A=0,
\]
and the complex conjugate of this equation. These two equations can be integrated to find $A$ in terms of $f$ and its derivatives to be
\be
\label{A0}
A=-m\frac{f_{\z\zb}}{\p_{\zb}f+\p_{\zb}\overline{C}},
\ee
where $C=C(\z)$ is an arbitrary holomorphic function. The freedom of adding a holomorphic function and its conjugate to $f$ can be used to set $C$ to zero.
This choice yields
the form of $\xb$ in (\ref{gandV}). 
Finally, using  the formula for the Ricci scalar
\begin{equation}
R=-(\Delta_0 f)^{-1} \Delta_0(\ln{(\Delta_0 f)}), \quad\mbox{where}\quad\Delta_0=4\frac{\p^2}{\p\zeta\p\bar{\zeta}}  \label{eq:flatlap}
\end{equation} 
is the flat Laplacian,
the  $(\zeta\zb)$
component of (\ref{QEE}) gives the 4th order PDE (\ref{pde}) for $f$.
\end{proof}
If the gradient and Laplacian
are taken with respect to the Cartesian coordinates 
$(\frac{1}{2}(\z+\zb), \frac{1}{2i}(\z-\zb))$ then,
expanding the derivatives and isolating the numerator in (\ref{pde}), we find that 
this equation is equivalent to a vector PDE on $\R^2$ given by
\be
\label{vector_pde}
\frac{1}{m}{\Delta_0}^2 f-\frac{1}{m}\frac{|\nabla(\Delta_0 f)|^2}{\Delta_0 f}+
2\frac{(\Delta_0 f)^3}{|\nabla f|^2}+\Delta_0 f\nabla f\cdot
\nabla\Big( \frac{\Delta_0 f}{|\nabla f|^2} \Big)+\frac{2}{m}\lambda(\Delta_0 f)^2=0.
\ee
Here $\nabla$ and 
$\Delta_0=\nabla\cdot\nabla$ 
are the nabla operator and the Laplacian of the flat Euclidean metric.

\section{Symmetry reduction of the quasi-Einstein equation}
\label{sectionlocal}
In this section we shall show that if $g$ is assumed to be axisymmetric 
then the 4th order ODE resulting from the PDE (\ref{pde}) is solvable by quadratures. 

\subsection{Inheritance of symmetry}
We shall first demonstrate that if  metric admits a Killing vector field, the constraints (\ref{step1}) and (\ref{step2}) resulting from the prolongation procedure can be used to deduce that $X$ inherits the symmetry from $g$. 
\begin{prop}
\label{prop_inh}
Let $(g, X)$ be a solution to the quasi-Einstein equations (\ref{QEE}) on a two-dimensional connected surface $M$ admitting a Killing vector field $K$. Then either $[K, X] = 0$ or $g$ has constant curvature.
\end{prop}
\begin{proof}
We start with the hardest case in which $R$ is not constant, $m \neq -4$ and $(m, \lambda) \neq (-1,0)$. By real analyticity (Lemma \ref{RAlem}) $R$ can not be constant on an open set. It is therefore sufficient to prove that $K$ and $X$ commute in a neighborhood of every point where $(m+1)R \neq 2\lambda$ and d$ R \neq 0$. Let us choose coordinates such that $g = \td r^2 + k(r)^2\td\theta^2$ with $k(r) > 0$ and $K = \partial_\theta$. Writing $X^\flat = X_r \td r + X_\theta \td\theta$, equations (\ref{step1}) and (\ref{step2}) are three polynomials in $X_r, X_\theta$ and $\Omega$ with coefficients depending on $r$ only. We use this to show that the components of $X^\flat$ do not depend on $\theta$, which is equivalent to $[K, X]=0$.

We start by solving (\ref{step1}) for $X_\theta^2$ to find
\begin{equation} \label{xt2}
X_\theta^2 = k^2\left(\frac{m^2\Delta R-m(m+4)X_rR'-3m\Omega^2}{2(2\lambda - (m+1)R)}+m\lambda - \frac 12mR-X_r^2\right),
\end{equation}
where the prime denotes a derivative with respect to $r$. The $r$- and $\theta$-components of (\ref{step2}) are
\begin{align}
0&= -(\Delta R)'+\tfrac{m+4}{m^2}X_r^2R' + \tfrac{m+4}{m}X_rR''+R'\left(-\tfrac{2(m+1)}{m^2}(X_\theta^2k^{-2}+X_r^2) - (\tfrac 52 + \tfrac 4m)R + \lambda(3 + \tfrac 8m)\right) \nonumber\\
&\hspace{.4cm}+ \tfrac{18}{m^2}\Omega^2X_r + \tfrac{4}{m^2}(2\lambda - (m+1)R)\left(X_r(\lambda - \tfrac 12R) +2\Omega k^{-1}X_\theta + \tfrac 1mX_r(X_r^2 + k^{-2}X_\theta^2)\right),\label{rcomp}\\
0&= R'\left(\tfrac{m+4}{m^2}X_\theta X_r + \tfrac{m+4}{m}X_\theta k'k^{-1} - \tfrac{m+16}{2m}\Omega k\right) + \tfrac{18}{m^2}\Omega^2X_\theta \nonumber \\
&\hspace{.4cm} +\tfrac{4}{m^2}(2\lambda - (m+1)R)\left(X_\theta(\lambda - \tfrac 12R)-2\Omega kX_r+\tfrac 1mX_\theta(X_r^2 + k^{-2}X_\theta^2)\right).\label{tcomp}
\end{align}
We use (\ref{xt2}) to eliminate the $X_\theta^2$ terms in (\ref{rcomp}) and (\ref{tcomp}). After doing this there is only one occurrence of $X_\theta$ left in (\ref{rcomp}) in a term proportional to $X_\theta \Omega$. We can solve this for $X_\theta \Omega$ and use it to eliminate $X_\theta$ completely from (\ref{tcomp}) times $\Omega$. The resulting equation is polynomial in $u = \Omega^2$ and $Y = (m+4)R'X_r$. It can be written in the form
\begin{equation}\label{p1}
P_1(u,Y) = \alpha_0(Y) + \alpha_1(Y)u + \alpha_2(Y)u^2= 0,
\end{equation}
where $\alpha_0, \alpha_1, \alpha_2$ is a polynomial of degree $3, 2, 1$ with leading term $\frac{1}{8}Y^3, -3Y^2, 18Y$ respectively. The lower order coefficients are complicated expressions involving up to 2 derivatives of $k$ and up to 3 derivatives of $R$.

We obtain a second polynomial constraint in $u$ and $Y$ by contracting (\ref{step2}) with $X$ and using (\ref{xt2}) and the expression for $X_\theta \Omega$ obtained from (\ref{rcomp}) to eliminate any occurrences of $X_\theta$. This leads to
\begin{equation}\label{p2}
P_2(u,Y) = \beta_0(Y) + \beta_1(Y)u - 6u^2 = 0,
\end{equation}
where $\beta_0,\beta_1$ is a polynomial of degree $2,1$ with leading term $\beta_1 = -\frac14Y + \dots$ and 
\begin{equation}
    \beta_0 = \left[\frac 16 + \frac{5m}{16(m+4)} + \frac{2\lambda - (m+1)R}{3(m+4)R'}\left(\frac{R''}{R'}-\frac{k'}{k}\right)\right]Y^2 + \dots.
\end{equation}
The lower order terms again contain $k, k', k''$ and up to 3 derivatives of $R$. Let us view $P_1$ and $P_2$ as polynomials in $u$. The resultant res$(P_1, P_2)$ is a polynomial in $Y$ of degree at most 6 with coefficients depending on $r$ only. If there is a solution such that $X_r$ (and hence $Y$) depends on $\theta$, this polynomial must be identically zero. To calculate the coefficient $\gamma_6$ of $Y^6$ we only need the leading coefficients of $\alpha_i$ and $\beta_i$. Setting $\gamma_6$ to zero gives
\begin{equation}
R''  = \frac{k'R'}{k}-\frac{5(m+1)(R')^2}{4(2\lambda - (m+1)R)}.
\end{equation}
For the coefficient $\gamma_5$ of $Y^5$ we need the first subleading terms in $\alpha_i$ and $\beta_i$. Requiring $\gamma_5$ to vanish imposes
\begin{equation}
\left(R''  - \frac{k'R'}{k}+\frac{5(m+1)(R')^2}{4(2\lambda - (m+1)R)}\right)\delta -\frac{25}{3}(m+4)^2(R')^3(2\lambda - (m+1)R) = 0,
\end{equation}
where
\begin{align}
    \delta = \:&m(23m + 68)(m+4)(R')^2k'k^{-1}-15m(m+4)^2R'R'' +32m(2\lambda - (m+1)R)(\Delta R)' \nonumber\\
    &-16(3m+4)(m+1)(m+4)(R')R^2+32(5m^2+34m+32)\lambda R(R')+48m(m+1)R'\Delta R \nonumber \\
    &-128(m+8)\lambda^2R'-\tfrac{75}{4}m(m+1)(m+4)^2(2\lambda - (m+1)R)^{-1}(R')^3.
\end{align}
Thus given our assumptions the coefficients of $Y^5$ and $Y^6$ can not vanish simultaneously. Hence res$(P_1,P_2)$ is not identically zero, and so we must have $\partial_\theta X_r = 0$. But now (\ref{p2}) is an equation for $u$ with coefficients only depending on $r$, implying $\partial_\theta u = 0$. From (\ref{xt2}) we then also find $\partial_\theta X_\theta = 0$, and therefore $[K, X]=0$.

It remains to consider the cases $(m,\lambda) = (-1, 0)$ and $m = -4$. If $(m, \lambda) = (-1, 0)$ we can use (\ref{step1}) to eliminate the $\Omega^2$ term in (\ref{rcomp}). The resulting equation is
\begin{equation} \label{m=-1}
0 = -(\Delta R)' + \frac 32RR'  - 3X_r (R'' + 2\Delta R)  - 15X_r^2R'.
\end{equation}
It follows that $\partial_\theta X_r = 0$ because the coefficients of (\ref{m=-1}) depend on $r$ only and are not all zero. With the same argument as before equation (\ref{step1}) implies $\partial_\theta \Omega  = 0$ and (\ref{tcomp}) gives $\partial_\theta X_\theta = 0$ unless $\Omega = 0$, in which case we see from (\ref{pqe2}) that the curvature is constant.

If $m = -4$ we can use (\ref{step1}) to express $\vert X \vert^2$ in terms of $\Omega^2$. Using this the equations (\ref{step2}) become linear in $X^\flat$. They can be solved for $X^\flat$ in terms of $\Omega$, giving an expression of the form
\begin{equation} \label{m=-4}
\rho(\Omega)X^\flat = \sigma(\Omega),
\end{equation}
where $\rho$ is a non-negative function and $\sigma$ is a 1-form (these are also functions of $k, R$ are their derivatives). The function $\rho$ and (the components of) $\sigma$ are polynomials of degree up to $4$ in $\Omega$ with leading terms $\Omega^4$ and $-3(2\lambda + 3R)^{-1}(\text{d} R) \Omega^4$ respectively. The coefficients of $\rho, \sigma_r$ and $\sigma_\theta$ depend on $r$ only. Taking the norm of (\ref{m=-4}) and plugging this into (\ref{step1}) multiplied by $\rho(\Omega)^2$ gives a polynomial constraint in $\Omega$ with leading term $-\frac 34\Omega^{10}$ coming from the $\frac{3}{m}\Omega^2$ term in (\ref{step1}). It follows that $\partial_\theta \Omega = 0$. Finally,  if $\Omega \neq 0$ then $\rho>0$ and hence Lie deriving \eqref{m=-4} along $K$ shows that $\mathcal{L}_K X^\flat=0$, whereas if $\Omega=0$ this follows by Lie deriving  (\ref{pqe2}) along $K$.  

We conclude that whenever $R$ is not constant we have $[K, X] = 0$.\end{proof}

\begin{remark} In the constant curvature case  $X^\flat$ is closed (see Remark \ref{rem_const_R}). Then, locally there is always a Killing vector $K$ commuting with $X$. If locally $X^\flat = \text{d}h$, then $K^\flat = \star \text{d} (e^{-h/m})$.
\end{remark}

\subsection{General axi-symmetric solution}
We are now ready to present the local normal form of the most-general axi-symmetric quasi-Einstein structure in two dimensions.
\begin{theorem}
\label{theorem_normal_K}
 If a solution $(M, g, X)$ to the quasi-Einstein equation (\ref{QEE}) on a two-dimensional surface $M$
admits an axial Killing vector field $K$, then there exist coordinates $(x, \phi)$ on $M$, and a function $B=B(x)$ such that
the quasi-Einstein structure is of the form
\be
\label{B_met_1}
g=B(x)^{-1}\mathrm{d}x^2+B(x)\mathrm{d}\phi^2, \quad 
\xb=\frac{-m}{x^2+\beta^2}\Big(x\mathrm{d} x-B(x)\beta \mathrm{d}\phi\Big),
\ee
where $\beta=\mbox{const}$, and $B$ is given by
\be
\label{hyper_geom_1}
B(x)=
\begin{cases}
b x(x^2+\beta^2)^{-m/2}+c\;  F(x/\beta)(x^2+\beta^2)^{-m/2}
-\frac{\lambda(x^2+\beta^2)}{m+1} & \text{for $m\neq -1, \beta\neq 0$}\\
     x\Big(b-\lambda\arcsinh(x/|\beta|)\Big)\sqrt{x^2+\beta^2}+c(x^2+\beta^2)        & \text{for $m=-1, \beta\neq 0$}\\
     bx^{1-m}+c-\frac{\lambda x^2}{m+1}      & \text{for $m\neq\{-1,  1\},  \beta= 0$}\\
       b\ln{x}+c-\frac{\lambda x^2}{2}       & \text{for $m=1, \beta = 0$}\\
        cx^2+b-\lambda x^2\ln{x}       & \text{for $m=-1, \beta = 0$.}
\end{cases}
\ee
Here $b, c$ are arbitrary constants and $F(x) := _2 F_1 (-\frac{1}{2},-\frac{m}{2}, \frac{1}{2}, -x^2)$ is the hypergeometric function. 
\end{theorem}
\begin{proof} Recall Proposition \ref{prop_inh} guarantees that $[K, X]=0$.
We shall first establish that the coordinates can be chosen so that $K(f)=\mbox{const}$, where $f$ is the 
K\"ahler potential from Proposition \ref{prop1}. 
We can rewrite $\xb$ in (\ref{gandV}) in the form
\begin{equation}
\xb= -\tfrac{1}{m} |X|^2 \td f .  \label{eq:Xf}
\end{equation} 
Therefore, $\mathcal{L}_K X=0$ implies invariance of  $\td f$ along a Killing field $K$.  
Conversely, we can invert the previous relation so $X^\flat= -m\;\td f/|\td f|^2$, which implies that if $\td f$ is Lie-derived by a Killing field $K$ then so is $X^\flat$.  We conclude that  $[K, X]=0$ iff ${\mathcal L}_K(\td f)=\td (K(f))=0$ or equivalently if $K(f)$ is a constant. 

Choose  the holomorphic coordinate $\zeta=re^{i\phi}$ on $M$ (so that the $U(1)$-action is $\zeta\rightarrow e^{i\psi}\zeta$), we see that the metric takes the form (\ref{gandV}) and the K\"ahler potential satisfies
\be
\label{Kandf}
K(f_{\zeta\zb})=0, \quad\mbox{where}\quad
K=i(\zeta\p_\zeta-\zb\p_{\zb}).
\ee
The solution to this is 
\[
f=\rho(r)+R(\zeta)+\bar{R}(\zb), \quad \mbox{where}\quad r=|\zeta|, 
\]
and $\rho(r)$, $R(\z)$, $\bar{R}(\zb)$ are integration functions. Moreover $[K, X]=0$ iff $K(f)=\mbox{const}$ which gives
\[
\z R_{\z}=\frac{1}{2}(\alpha+i\beta)
\]
for some constants $(\alpha, \beta)$.

Now set $r=e^s$ and  $x=d \rho/d s$, to find that
\[
g=\frac{\td x}{\td s}\Big(\td s^2+\td\phi^2), \quad
\xb=-\frac{m}{4}\frac{\td x}{\td s}\Big(\frac{\td s+i\td \phi}{x/2+\zb \bar{R}_{\zb}} + \frac{\td s-i\td \phi}{x/2+\zeta 
R_{\zeta}}\Big).
\]
Setting 
\be
\label{def_of_B}
B=\Big(\frac{\td s}{\td x}\Big)^{-1}
\ee
puts $(g, X)$  in the form (\ref{B_met_1}), where the constant $\alpha$ has been set to $0$ by translating the function $x$.
The PDE (\ref{pde})
reduces to an ODE 
which linearises under the reciprocal transformation (\ref{def_of_B}). The resulting equation
for $B=B(x)$ is
\be
\label{b_eq}
B''+\frac{m}{(\beta^2+x^2)^2}\Big(x(\beta^2+x^2)B'+2\beta^2B\Big)+2\lambda=0, \quad \mbox{where}\quad \beta=\mbox{const}
\ee
which can be solved explicitly by  (\ref{hyper_geom_1}).
\end{proof}

If $X^\flat$ is not closed, which requires $\beta\neq 0$, then without loss of generality we can set $\beta=1$. To see this, set $x=\beta{\tilde{x}}, \phi=\beta^{-1}\tilde{\phi}, B=\beta^2\widetilde{B}$. Rewriting
(\ref{B_met_1}) and (\ref{b_eq}) in terms of $(\widetilde{B}, \tilde{x}, \tilde{\phi})$ and dropping tilde has an effect of setting $\beta=1$, and reducing  (\ref{hyper_geom_1}) to the form (\ref{hyper_geom}).  However, with such a choice the coordinate $\phi$ is no longer $2\pi$-periodic.

\begin{remark}
One can also derive the  local form of the solution \eqref{B_met_1} and \eqref{hyper_geom_1} by generalising the method used for the $m=2$ case in the context of vacuum near-horizon geometries~\cite[Theorem 4.3]{KL13}. If $X^\flat$ is not closed, so we can set $\beta=1$, we can write the Killing field of (\ref{B_met_1}) as $K^\flat= \Gamma X^\flat+ \frac{m}{2} \td\Gamma$ where $\Gamma=(1+x^2)/m$ is a smooth positive function.  This generalises the form of the Killing field for the $m=2$ case, that was proven must exist on  $S^2$~\cite{DL23}. In the case $m=2, \lambda=0$ the local form of \eqref{B_met_1}, as well as its extension to a two--sphere was found in \cite{LP1}
in the context of extremal horizons.

  \end{remark}

\section{Extension to a two-sphere}
\label{sectionregular}
We will make use of the following standard result. We include a proof for completeness.
\begin{lemma}
\label{lem:gsmooth} 
The metric (\ref{B_met}) extends to a smooth metric on $S^2$ if and only if there exist adjacent simple zeros $x_1<x_2$ of $B(x)$ such that $B(x)>0$ for all $x_1<x<x_2$ and
\begin{equation}
B'(x_1)=-B'(x_2)   \label{eq:smooth}
\end{equation}
where $\phi \sim \phi+p$ is periodically identified with period $p= 4\pi/|B'(x_i)|$.    
\end{lemma}
\begin{proof} 
First note that any Killing field on $S^2$ must have closed orbits.
The metric  (\ref{B_met}) is valid in a chart $(x, \phi)$ such that $B(x)>0$. In this chart $K= \partial_\phi$ is a Killing field and $|K|^2= B$. Since the orbits of $K$ are closed the coordinate  must be periodically identified $\phi\sim \phi+p$ for some period $p>0$.  Let $(x_1, x_2)$ be the maximal interval such that $B(x)>0$ for all $x\in (x_1, x_2)$ (including possibly an infinite or semi-infinite interval).  Requiring that $g$ extends to a smooth metric on a compact surface $M$  implies that $(x_1, x_2)$ must be a finite interval, so by maximality and continuity of $|K|^2$ we must have $B(x_1)=B(x_2)=0$. Thus $|K|^2=B>0$ for $x\in (x_1, x_2)$ and $ |K|^2\to 0$ as $x\to x_1$ or $ x\to x_2$, so we may add points corresponding to $x=x_1$ and $x=x_2$ to obtain a compact surface $M$ (these could correspond to the same point). These points are zeros of the Killing field $K$.

For the fixed point $x=x_1$ we can introduce geodesic coordinates $(s, \phi)$ by $s := \int_{x_1}^x \td x/\sqrt{B(x)}$ in which the metric takes the form $g= \td s^2+ f(s) \td \phi^2$ where $f(s):= B(x(s))$.  In general, the metric has a conical singularity at $s=0$ which is absent iff, for closed constant-$s$ curves, the ratio of the circumference to  radius approaches $2\pi$ as $s\to 0$, that is, 
\begin{equation}
2\pi = \lim_{s\to 0}  \frac{p \sqrt{f(s)}}{s}  = p  \lim_{s\to 0} \frac{\td}{\td s} \sqrt{f(s)} = \frac{p}{2} B'(x_1) \qquad \implies \qquad  p= \frac{4\pi}{B'(x_1)} \; , 
\end{equation}
where $p$ is the period of an orbit of $K=\partial_\phi$.    An identical argument can be applied to the fixed point $x=x_2$ by defining $s:= \int_x^{x_2} \td x/\sqrt{B(x)}$  which shows that the conical singularity at $s=0$ is absent iff $p =- 4\pi / B'(x_2)$.

The vector field $K$ has one or two simple zeros so $M=\RP^2$ or $S^2$ by the Poincar\'e-Hopf theorem. If we assume that $M$ is orientable, then
it is diffeomorphic to $S^2$, which can independently
be verified by the Gauss-Bonnet theorem:
The scalar curvature is $R=-2 B''(x)$ and hence
\be
\label{GB_calculation}
\frac{1}{4\pi} \int_M R \; \mbox{vol}_M = -\frac{p}{4\pi} [B'(x)]_{x_1}^{x_2}= 2  \; ,
\ee
so the surface is $M=S^2$.
\end{proof}

\begin{remark}
In the case of $\RP^2$ the two zeros of $B$ are identified by the antipodal map, and correspond to the same point on $M$. In (\ref{GB_calculation}) the LHS has to be halved to take the $\Z_2$ quotient into account. In
fact, all global solutions of Theorem \ref{main_theo_md} descend to $\RP^2$.
\end{remark}
\subsection{Hypergeometric function}
To investigate the regularity conditions imposed by  Lemma \ref{lem:gsmooth}  on the local form
(\ref{B_met}) we will need a couple of results about the hypergeometric function.
\begin{lemma}
The function $F(x) := _2 F_1(-1/2,-m/2, 1/2, -x^2)$ admits an integral representation
\be
\label{formula_for_F}
F(x)=1-x\int_0^x \frac{(y^2+1)^{\frac{m}{2}}-1}{y^2}\; \mathrm{d} y.
\ee
It is monotonically increasing if $x>0$  and $m<0$ and monotonically decreasing for $x>0$ if $m>0$.
\end{lemma}
\begin{proof}
The key observation is that  for all $m\neq 0$
\be
\label{Hyp_geom_id}
\frac{\td}{\td x} \Big( x^{-1} F(x)\Big)  =-x^{-2}(1+x^2)^{m/2}.
\ee
Moreover, from the hypergeometric series we have $F(x)=1+O(x^2)$. Therefore integrating 
(\ref{Hyp_geom_id}) gives the integral representation (\ref{formula_for_F}).
Noting that $(y^2+1)^{m/2}$ is greater (smaller)
than $1$ if $m$ is positive (negative) establishes the monotonicity claim.
\end{proof}
The representation (\ref{formula_for_F}) implies that 
the asymptotic form
of $F(x)$ as $x\rightarrow \infty$ is
\be
\label{asymptoticF}
F(x) \sim
\begin{cases}
  \sqrt{\pi}\frac{\Gamma\Big(-\frac{m}{2}+\frac{1}{2}\Big)}{\Gamma\Big(-\frac{m}{2}\Big)}x& \text{for $m\in (-\infty, 0)\cup(0, 1)$}\\
-x\ln{x} & \text{for $m=1$}\\
-\frac{1}{m-1}x^m  &\text{for $m>1$.}
\end{cases}
\ee
The coefficient of the linear term is negative for $m\in (0, 1)$ and positive for $m<0$.

The next lemma shows the existence of two simple zeros of $F$ for $m>0$.
\begin{lemma}
\label{lemmahc2}
For  $m>0$ the function $F$ is monotonically decreasing faster than a linear function: 
For any $x_*>0$ there exists a constant $a(x_*)$ such that $F(x)<1-a(x_*)x$ if $x>x_*$. In particular $F$ has precisely two simple zeros.
\end{lemma}
\begin{proof}
It is convenient to define 
 \[ f(t)= \frac{(1+t^2)^{m/2} - 1}{t^2}  \]
 which for $m>0$  is a strictly positive function for $t>0$.  Then we can write $F(x)= 1 - x \int_0^x f(t) \td t$. Note that $F(0)=1$ so $F(x)>0$ near $x=0$.  
  Positivity of $f(t)$ implies that for $y>x>0$ we have $\int_0^y f(t) \td t > \int_0^x f(t) \td t$.  Thus $F(x)$ for $m>0$ is monotonically decreasing for $x>0$. However, we need a little more.  In particular, fix $x_*>0$, then for all $x>x_*$
 $$
  x \int_0^x f(t) \td t >  x \int_0^{x_*} f(t)  \td t
  $$
  and hence
  $$
  F(x)<1- x\int_0^{x_*} f(t)  \td t  \;.
  $$
  Thus $F(x)$ for $m>0$ and $x>x_*$ is monotonically decreasing faster than a linear function.  
  
  Now assume for contradiction that $F(x)>0$ for all $x>0$.  But then  for all $x>x_*$ we have  that $x \int_0^{x_*} f(t)  \td t <1$ which is a contradiction.  Therefore it must be the case that there exists a zero of $F(x)$ at say $x=x_2>0$. Furthermore, monotonicity of $F(x)$ means that this is the only zero for $x>0$. Finally, recall that $F(x)$ is an even function and therefore it must have precisely two zeros at $x=\pm x_2$.  The fact they are simple follows from evaluating \eqref{Hyp_geom_id} at the roots which gives
  $$
  F'(x_i)= - \frac{( 1 +x_i^2)^{m/2}}{x_i}
  $$
  and so $F'(x_2)=-F'(x_1)<0$.
\end{proof}

\subsection{The parity of $B$}
The function $B$ appearing in the local normal form (\ref{B_met}) is a sum of an even part given by the multiple
of the hypergeometric function, and an odd part. We shall now show that the regularity conditions of Lemma \ref{lem:gsmooth} force $B$ to be even.
\begin{lemma}
\label{parity_lemma}
The metric (\ref{B_met}) extends to a smooth metric on $M=S^2$  if and only if the function $B$ given by (\ref{hyper_geom}) is even (and so $b=0$).
\end{lemma}
\begin{proof} First observe that the function $B$ given by (\ref{hyper_geom}) satisfies the following first order ODE
\be
\label{Bode}
\frac{\td}{\td x} \left( \frac{B(x)(x^2+1)^{m/2} }{x}\right) =- \frac{(1+x^2)^{m/2}}{x^2} ( \alpha +\lambda x^2), \qquad \alpha:= \begin{cases}c- \frac{\lambda}{m+1} & \text{for $m\neq -1$} \\
c & \text{for $m=-1$} \end{cases}
\ee
where to evaluate the derivative we used the identity (\ref{Hyp_geom_id}).    We also note that $B(0)=\alpha$.

    Let $x_1 < x_2$ be two adjacent zeros of $B$ such that $B'(x_1) = -B'(x_2) = \delta > 0$. Assume first that $\lambda \neq 0$ and $m\neq -1$  and set
\begin{equation}
    \hat{F}(x) := \frac{B(x)(x^2 +1)^{\frac m2}}{\lambda x}  \; . 
\end{equation}
Setting $\gamma = -\frac{B(0)}{\lambda}$ and using (\ref{Bode}) shows that 
\begin{equation} \label{Fprime}
    \hat{F}'(x) = -\frac{(x^2+1)^{\frac m2}(x^2 - \gamma)}{x^2}.
\end{equation}
If $x_1$ and $x_2$ are non-zero we can write
\begin{equation} \label{Fhatreg}
    \hat{F}'(x_1) = \frac{\delta(x_1^2 +1)^{\frac m2}}{\lambda x_1} \hspace{.2cm}\text{and}\hspace{.2cm} \hat{F}'(x_2) = -\frac{\delta(x_2^2 +1)^{\frac m2}}{\lambda x_2}.
\end{equation}
If $x_1=0$ then $B(x_1)=0$ and since we require $B$ to be smooth at $x_1$,  then $B(x)/x$ must a smooth function and hence $\hat{F}(x)$ is well defined at $x=0$ too. Therefore $\hat{F}$ is well-defined on all of 
$\mathbb{R}$ and is strictly decreasing with $\hat{F}(0) = \frac{\delta}{\lambda}$ and $\hat{F}'(x) = -(x^2 + 1)^{\frac m2}$.

 Equation (\ref{Fhatreg}) is still valid for $x_2$ and gives $x_2 = \frac{\delta}{\lambda}$. Since $x_2$ has to be positive this already excludes $\lambda < 0$. Moreover, depending on the sign of $m$, $\hat{F}'(x)$ is either always greater than $-1$ or less than $-1$. In both cases we can not have $\hat{F}(x_2) = 0$. Similarly, if $x_2 = 0$ we have $\hat{F}(0) = -\frac{\delta}{\lambda}$ and $x_1 = -\frac{\delta}{\lambda}$, but $\hat{F}$ decreases either too fast $(m > 0)$ or too slow $(m < 0)$ for this to happen. Hence $x_1$ and $x_2$ are both non-zero.

Eliminating $\delta$ between the two equations in (\ref{Fhatreg}) leads to
\begin{equation}
    \left(1 - \frac{\gamma}{x_1x_2}\right)(x_1 + x_2) = 0.
\end{equation}
There are now three cases to consider:
\begin{itemize} 
\item If $\lambda = 0$ repeating the calculation above with $\hat{F}$ replaced by $\lambda\hat{F}$ shows that we must have $x_1 = -x_2$, and $b=0$.
\item If $\lambda<0$ and  $x_1 = -x_2$ we find $b = 0$ and the result follows, so  let us instead
assume $x_1x_2 = \gamma$. If  $\gamma < 0$ we have $B(0) < 0$ and  the zeros have opposite sign, but $B$ is not positive between the zeros. If $\gamma > 0$ the zeros have the same sign and from (\ref{Fhatreg}) we need $\hat{F}' (x_2)$ to have the same sign as $x_2$. Since $x_1 < x_2$ we either have  $x_2 > \sqrt{\gamma}$ or $-\sqrt{\gamma} < x_2 < 0$. However, from (\ref{Fprime})  we find that  $\hat{F}'(x_2)$ and $x_2$ have opposite signs which is a contradiction.
\item If $\lambda>0$, we again assume that $x_1x_2 = \gamma$ as otherwise $x_1=-x_2$ and $b=0$.
Define
\begin{equation}
    G(x) := \hat{F}(x)  - \hat{F}\left(\frac{\gamma}{x}\right).
\end{equation}
Note that $G(x_1) = G(x_2) = 0$. We have
\begin{equation} \label{Gprime}
    G'(x) = \hat{F}'(x) +\frac{\gamma}{x^2}\hat{F}'\left(\frac{\gamma}{x}\right) = \left(\frac{\gamma}{x^2}-1\right)\left((x^2+1)^\frac m2 - \left(\frac{\gamma^2}{x^2} + 1\right)^{\frac m2}\right).
\end{equation}
The function $G'(x)$ vanishes if and only if $x^2 = \vert\gamma\vert$. If $\gamma > 0$ we have $G(\pm \sqrt{\gamma}) = 0$. From (\ref{Gprime}) we know that $G$ is monotonic between $x_2$ and sgn$(x_2)\sqrt{\gamma}$, contradicting the fact that $G(x_2) = 0$.

If $\gamma < 0$ we must have $x_1 < 0$ and $x_2 > 0$ and $B(0) > 0$. We will show that $G$ has no zeros in this case.

If $\gamma < 0$ and $m > 0$ the function $G$ has a maximum at $x = \sqrt{-\gamma}$. Since $x_2 > 0$ in order for $G(x_2) = 0$ to be possible we need $G(\sqrt{-\gamma}) \geq 0$. To show that this is not the case, set $u := \frac c\lambda > \frac{1}{m+1}$ and consider the function
\begin{align}
H(u) & := \frac{1}{2u}\left(\hat{F}(\sqrt{-\gamma}) - \hat{F}(-\sqrt{-\gamma})\right)  \\ &= \frac{1}{\sqrt{u - \tfrac{1}{m+1}}}\left(F\left(\sqrt{u - \tfrac{1}{m+1}}\right) - \frac{(u+1-\tfrac{1}{m+1})^{\frac m2 + 1}}{u(m+1)}\right).  \nonumber
\end{align}
Note that $H$ tends to zero as $u \to \frac{1}{m+1}$. A calculation using  (\ref{Hyp_geom_id})  shows
\begin{equation}
    H'(u) = -\frac{m(1+u-\tfrac{1}{m+1})^{\frac m2}\sqrt{u - \tfrac{1}{m+1}}}{(m+1)u^2}.
\end{equation}
We see that $H$ is strictly decreasing, implying that $G(\sqrt{-\gamma}) = 2uH(u) < 0$.

If $\gamma < 0$ and $m \in (-1, 0)$ we find that $G$ has a minimum at $x = \sqrt{-\gamma}$, so in order for $G$ to have zeros we require $G(\sqrt{-\gamma}) \leq 0$. However, this is not possible since $H$ is now increasing and therefore positive for $u > \frac{1}{m+1}$.

If $\gamma < 0$ and $m < -1$ the function $G$ again has a minimum at $\sqrt{-\gamma}$ and $H$ is decreasing. For $u < 0$ this implies $H$ is negative and so $G(\sqrt{-\gamma}) > 0$. For $u = 0$ direct evaluation shows $G(\sqrt{-\gamma}) > 0$. If $u$ is positive, so is $H$ because using 
(\ref{asymptoticF}) it tends to a positive constant as $u \to \infty$. We find $G(\sqrt{-\gamma}) > 0$, i.e. $G$ has no zeros.
\end{itemize}
Finally, if $m = -1$ imposing the conditions from Lemma \ref{lem:gsmooth}  again leads to $x_1 = -x_2$ with similar arguments as above, indeed, the only part that requires modification is the $\lambda>0, \gamma<0$ case.
\end{proof}
Apart from establishing the even parity of $B$ the argument above gives the following.
\begin{cor}
\label{cor_parity}
Let $x_1, x_2$ with $x_1<x_2$ be the adjacent zeros of $B$ as in Lemma \ref{lem:gsmooth}. Then 
$x_1=-x_2$ and in particular $B(0)>0$.
\end{cor}
\begin{prop}
\label{prop_zeros}
The function $B$ given by (\ref{hyper_geom}) and $b=0$ has two simple zeros and is positive between these zeros  
 iff  
\begin{center}
    \begin{tabular}{c|c|c|c}
   & $\lambda=0$ & $\lambda>0$ &$\lambda<0$\\
   \hline
 $m>0$ & $c>0$ & $c>\frac{\lambda}{m+1}$ &$c>\frac{|\lambda|}{m+1} c_0$\\  
 $m\in(-1,0)$ & - & $c>\frac{\lambda}{m+1}$ &-\\  
 $m=-1$ & - &  $c>0$ & -\\
 $m<-1$  & - & $c\in\Big(\frac{\lambda}{m+1},0\Big)$ &-
    \end{tabular}
\end{center}
 where\footnote{The asymptotic formula (\ref{asymptoticF}) implies that the minimum exists.}
\[
c_0=\min_{x>x_0} \frac{(x^2+1)^{\frac{m}{2}+1}}{|F(x)|}
\]
and $x_0$ is the unique positive zero of $F$.
\end{prop}
\begin{proof}
We shall prove Proposition \ref{prop_zeros} by analysing all
cases separately. It is convenient to introduce  $\hat{B}(x):= (1+x^2)^{m/2}B(x)$ which has the same zeros and sign as $B$.

\subsection*{ Case 1. $m>0, \lambda\geq 0$} From the formula (\ref{hyper_geom})  we have
$B(0)=c-\lambda/(m+1)$, so by Corollary \ref{cor_parity} 
\[
c>\frac{\lambda}{m+1}.
\]
In particular, $c>0$.  The zeros of $B$ coincide with those of 
\[
 \hat{B}(x) = c F(x)-  \frac{\lambda}{(m+1)} (x^2+1)^{\frac{m}{2}+1}.
\]
Note that $\hat{B}(0)>0$, and $\hat{B}(x)\leq c F(x)$
for all $x>0$. Therefore, by Lemma \ref{lemmahc2},  it follows that $B$ has precisely two roots, and $B$ is positive between these roots. A representative graph of $B$ is given in the figure below.
\begin{center}
{\includegraphics[width=4.9cm,height=4.9cm,angle=0]{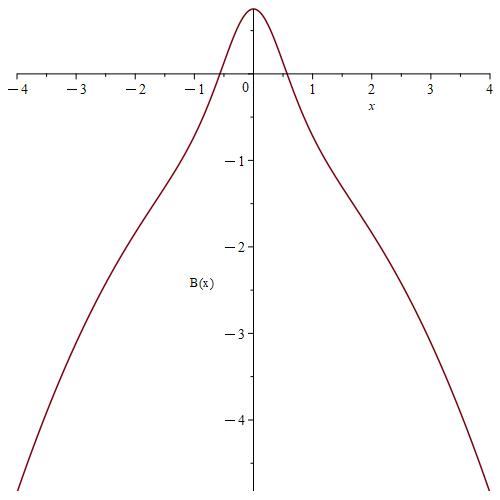}}\\\nopagebreak
{$m=3, c=1, \lambda=1$.}
\end{center}
\subsection*{Case 2. $m>0, \lambda<0$} In this case
\[
B(0)=c+\frac{|\lambda|}{m+1} \; ,
\]
and we require $B(0)>0$.
The zeros of $B$ coincide with those of 
\[
\hat{B}(x)= c F(x)+
\frac{|\lambda|}{(m+1)}(x^2+1)^{\frac{m}{2}+1}.
\]
If $c \leq 0$ then $\hat{B}(x)$ for $x>0$ is a sum of  increasing functions (Lemma \ref{lemmahc2}) and since $\hat{B}(0)>0$ it follows that $B$ has no zeros.
Now assume that $c>0$ which guarantees $B(0)>0$. Let $x_0$ be the unique positive 
zero of $F$, so that $F(x)<0$ for all $x>x_0$. To ensure that 
$B(x)<0$
for some $x>x_0$, and hence that $B$ has a zero for some $x>x_0$, it is necessary and sufficient that
\be
\label{case2c}
c>\frac{|\lambda|}{m+1}\min_{x>x_0} C(x),
\quad
\mbox{where}\quad  C(x)=\frac{(x^2+1)^{\frac{m}{2}+1}}{|F(x)|}.
\ee
The function $C(x)$ clearly has a minimum, as it is continuous
for $x\in (x_0, \infty)$, and (using the asymptotic form
(\ref{asymptoticF}))
tends to $\infty$ as $x$ tends to $x_0$ or $\infty$. A representative graph of $B$ is given in the figure below.
\begin{center}
{\includegraphics[width=4.9cm,height=4.9cm,angle=0]{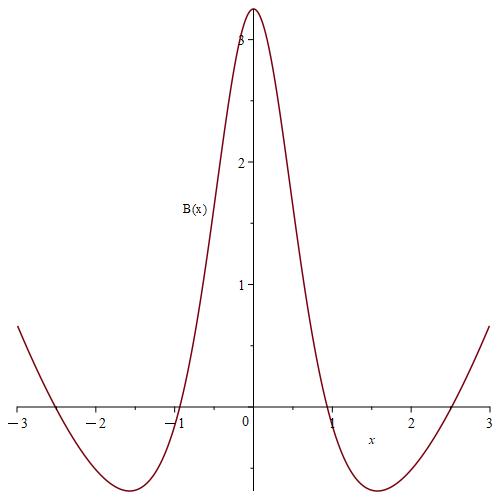}}\\\nopagebreak
{$m=3, c=3, \lambda=-1$.}
\end{center}

Finally we shall demonstrate that the two roots of $B$ are simple. Set $P=Bh$, where $h:(0, \infty)\rightarrow \R^+$. 
Then $B$ has a positive double zero at $x_2>0$ iff $P$ has a positive double zero at $x_2$.
Now take $h=(x^2+1)^{m/2}/x$. Then $P'(x)$ is given by the identity (\ref{Bode}), and we find that it vanishes only at $x_2=(B(0)/|\lambda|)^{1/2}$.
But $P$ tends to $\infty$ as $x$ tends to $0$ and $\infty$, and we have picked a $c$ such that 
$P$ is negative somewhere. Therefore its minimum must be negative. So it does not have a  double zero, and neither does $B$. 
\subsection*{Case 3. $m<0, \lambda>0$}
Reversing the arguments in the proof of Lemma \ref{lemmahc2}
shows that $F(x)$ is monotonically increasing faster that a linear function, and in particular that it tends
to $\infty$ from $F(0)=1$. There are now two cases to consider:
\begin{itemize} 
\item{\bf{Case 3a. $m<-1$}.}  The zeros of $B$ coincide with those of
\be
\label{hatF2}
\hat{B}(x)=c F(x) + \frac{\lambda}{|m+1|}(x^2+1)^{\frac{m}{2}+1} ,
\ee
and we have $\hat{B}(0)=B(0)>0$.
If $c\geq 0$ then $cF$ is increasing, and since $\hat{B}(x) > cF(x)$ this implies $\hat{B}$ and hence $B$ has no zeros.
Thus we assume that $c<0$.
Therefore $cF(x)$ is decreasing faster than a linear function, and $cF(0)=c<0$. 
We have that $\hat{B}(0)>0$ iff
\[
c\in\Big(\frac{\lambda}{m+1}, 0\Big).
\]
If $m<-1$ then $(x^2+1)^{\frac{m}{2}+1}<x$
for sufficiently large positive $x$. Therefore
the first term in (\ref{hatF2}) is decreasing faster than a linear function, while the second term is either not increasing or increasing slower than $x^{1-\epsilon}$ for some 
$\epsilon>0$. Thus $\hat{B}\rightarrow -\infty$
as $x\rightarrow\infty$, and in particular it has
two zeros. These are also the zeros of $B$. A representative graph of $B$ is given below.
\begin{center}
{\includegraphics[width=4.9cm,height=4.9cm,angle=0]{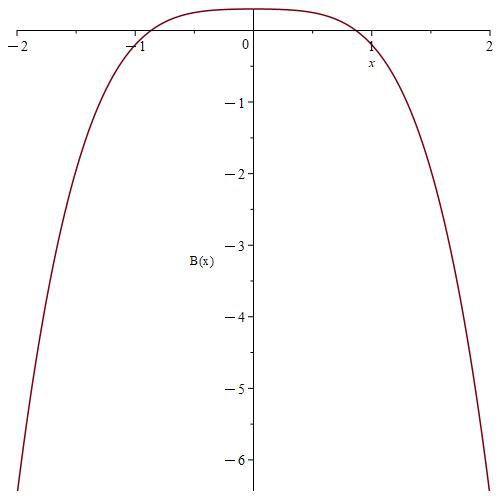}}\\\nopagebreak
{$m=-3, c=-1/5, \lambda=1$.}
\end{center}
\item{\bf{Case 3b. $m\in (-1, 0)$}.} We have $B(0)>0$ iff 
\be
\label{case3bb}
c>\frac{\lambda}{m+1}>0 \; .
\ee
The zeros of $B$ coincide with those of
\[
\hat{B}(x)= c F(x) - \frac{\lambda}{m+1} (1+x^2)^k,   \quad\mbox{where}\quad k=\frac{m}{2}+1\in\Big(\frac{1}{2}, 1\Big).
\]
Using the asymptotic form (\ref{asymptoticF}) we see that the first term in $\hat{B}$ grows linearly for large $x$, whereas the second term in $\hat{B}$  decreases as $x^{m+2}$, and therefore $\hat{B}\to -\infty$ as $x \to \infty$. Therefore, $\hat{B}$ and hence $B$ has a zero. A representative graph of $B$ is given below. 
\begin{center}
{\includegraphics[width=4.9cm,height=4.9cm,angle=0]{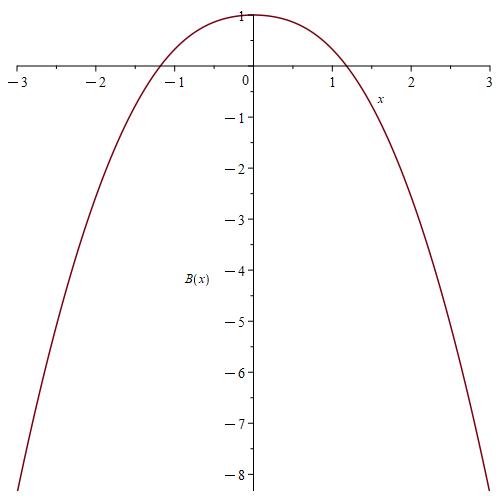}}\\\nopagebreak
{$m=-1/2, c=3, \lambda=1$.}
\end{center}
\end{itemize}
\subsection*{Case 4. $m=-1, \lambda>0$} 
Consider the explicit formula for $B$ in (\ref{hyper_geom}) corresponding to $m=-1$, and
take $c>0$. Therefore $F(0)=c>0$, and the existence two zeros follows as 
$x \arcsinh(x)$ grows faster that any constant multiple of $\sqrt{x^2+1}$. A representative graph is below.
\begin{center}
{\includegraphics[width=4.9cm,height=4.9cm,angle=0]{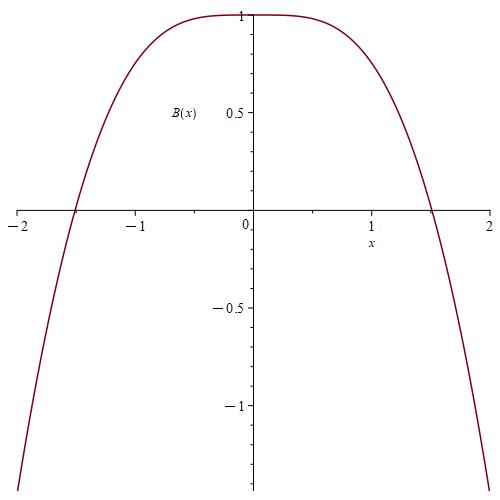}}\\\nopagebreak
{$m=-1, c=1, \lambda=1$.}
\end{center}
\end{proof}

We are now ready to prove our main theorem. \\

\noindent
{\em Proof of Theorem \ref{main_theo_md}.}
The Theorem now follows from the statements of Theorem \ref{theorem_normal_K}, Proposition
\ref{prop_zeros}, and Lemmas \ref{lem:gsmooth}
and \ref{parity_lemma} as follows:
Theorem \ref{theorem_normal_K} gives the local normal form (\ref{B_met}) and \eqref{hyper_geom} under an assumption that  $g$ has an axial Killing vector.  Lemma \ref{lem:gsmooth}  gives necessary and sufficient conditions on $B$ for the
metric to extend to $S^2$, and Proposition \ref{prop_zeros} then shows that $B$ given by 
(\ref{hyper_geom}) with $b=0$ satisfies the conditions this Lemma. 
Finally all axi-symmetric solutions on $S^2$ arise from (\ref{B_met}) and (\ref{hyper_geom}) with $b=0$ as a consequence of  Lemma \ref{parity_lemma} which showed that $B$ must be even. \koniec

\section{Projective metrizability and quasi-Einstein with $m=-1$}
\label{section_proj}
In this section we shall rule
out the existence of non-trivial solutions
to the quasi-Einstein equation (\ref{QEE}) with $m=-1$ and $\lambda =0$
on any compact orientable surface. 
\vskip 2pt
\noindent
{\em Proof of Theorem \ref{mminus1}.}
Let $(g, X)$ satisfy the quasi-Einstein equation
  (\ref{QEE}) on a surface $M$.
Define
an affine connection
$D$ (possibly with torsion) by
  \be
    D=\nabla -p\; \xb \otimes \mbox{Id} -q\; \mbox{Id}\otimes \xb
    \ee
    where $\nabla$ is the Levi-Civita connection of the metric $g$
    and $(p, q)$ are constants.
    The torsion of this connection is 
    $(p-q) \mbox{Id} \wedge \xb $, and the skew part of the Ricci
    tensor is $-\tfrac{1}{2}(2p+q)\td\xb$. 
    We verify by explicit
    calculation that the connection
    $D$ has vanishing Ricci tensor iff
    $m=-1, p=-\frac{1}{2}, q=1$. We now refer to a result of Milnor \cite{milnor} which states that
    if a  closed surface admits a flat affine connection, then $M$ is diffeomorphic to a Klein bottle, or a torus. The Klein bottle is not orientable, which leaves $M=T^2$. It now follows from Proposition \ref{prop2} that $X=0$
    and the metric $g$ is flat.
 \koniec
The importance of this result stems from projective differential geometry \cite{eastwood}, where it
is an open problem to determine all affine connections  with skew-symmetric
Ricci tensor which share the unparametrised geodesics with a Levi-Civita
connection of some (pseudo) Riemannian metric. 
\begin{definition}
A projective structure $[\nabla]$ is an equivalence class of torsion-free affine connections which share
the same unparametrised geodesics. Two affine connections $\nabla$ and $\hat{\nabla}$ belong to the same projective structure iff there exists $\Upsilon\in \Lambda^1(M)$ such that
\[
\hat{\nabla}=\nabla+\mbox{Id}\otimes \Upsilon+\Upsilon\otimes\mbox{Id}.
\]
A projective structure is called:
\begin{itemize} 
\item {\em metrisable} if it contains a Levi-Civita
connection of some metric in its projective class.
\item {\em skew} if it contains a 
representative connection with totally skew-symmetric Ricci tensor.
\end{itemize}
\end{definition}
The Schouten tensor of a connection $\nabla$ is defined in terms of the Ricci tensor
by
\[
\Rho_{ab}=\frac{1}{n-1}R_{(ab)}+\frac{1}{n+1}R_{[ab]},
\]
and is related to the Schouten tensor of $\hat{\nabla}$ by
\be
\label{schouten}
\hat{\Rho}_{ab}={\Rho}_{ab}-\nabla_b\Upsilon_a+\Upsilon_a\Upsilon_b.
\ee 
In case of two dimensional projective structures, the metrisability is obstructed by three
projective differential invariants \cite{BDE}, which are differential 
expressions of order at least five in components of a chosen connection in the projective class. 
The skew projective structures have been partially characterised by projective differential invariants \cite{randal,krynski,nur}. 
We aim to find metrisable connections projectively equivalent to skew connections.
Assume that $\nabla$ is the Levi-Civita connection of some metric $g$, and seek
a one-form $\Upsilon$ such that the Ricci tensor of $\hat{\nabla}$ is skew. Therefore $\hat{\Rho}_{(ab)}=0$, and the formula
(\ref{schouten}) with $\Upsilon=\frac{1}{1-n}X^{\flat}$ yields
\be
\label{Rpro}
R_{ab}=-\frac{1}{n-1}X_aX_b-\nabla_{(a}X_{b)}
\ee
which is the quasi-Einstein equation (\ref{QEE}) with $m=1-n$ and $\lambda=0$.
It is known \cite{derdzinski} that compact surfaces admitting connections with
skew-symmetric Ricci  are diffeomorphic to a torus or a Klein bottle. Theorem \ref{mminus1} directly shows
that the torus is necessarily flat if the connection is additionally metrisable. 

\section*{Appendix}
\appendix
\renewcommand{\theequation}{A.\arabic{equation}}
\renewcommand{\thetheorem}{A.\arabic{equation}}
\setcounter{equation}{0}
In \cite{DL23} it was shown that the quasi-Einstein-equations with $m=2$ on a closed $n$-dimensional manifold $M$, and with a non-gradient 
$X$ imply the existence of a Killing vector 
with a dual one-form  $K^{\flat}=\td \Gamma+\Gamma X^{\flat}$ for some smooth and positive function $\Gamma$. If additionally $n=2$ or $\lambda\leq 0$ then this Killing
vector necessarily commutes with $X$. This result was sufficient to prove the rigidity of the extreme Kerr horizon with or without the cosmological constant. In this appendix we shall use the Fredholm alternative
to show that $K$ commutes with $X$ in any dimension  for {\it any} sign of the cosmological constant. 
\begin{theorem}
\label{theoapp}
Let $(M, g, X)$ be a compact quasi-Einstein manifold
with $m=2$, and  $X^\flat$ non-gradient. 
Suppose that $(M, g)$ admits a Killing vector field of the form $K^{\flat}=\mathrm{d}\Gamma+\Gamma X^{\flat}$
for some smooth and positive function $\Gamma$. Then $[K, X] = 0$. 
\end{theorem}
\begin{proof}
Recall that (see Remark 2.6 in \cite{DL23})   
$[K, X] = 0$ if and only if $\mathcal{L}_K \Gamma = 0$. Consider the linear second order elliptic operator
\begin{equation}
    L\psi := -\Delta \psi     + \nabla_a( (\Gamma^{-1}\nabla^a\Gamma) \psi)+ \Gamma^{-2}|K|^2  \psi
\end{equation}
with $\psi \in C^\infty(M)$. We will show that the kernel of $L$ is trivial and $L(\mathcal{L}_K\Gamma) = 0$, which implies the result.

Lie deriving the equation $0 = \nabla_a K^a = \nabla_a(\Gamma X^a) + \Delta \Gamma$ along $K$ using the fact that $K$ is Killing gives
\begin{equation} \label{lie}
    0 = \Delta (\mathcal{L}_K\Gamma) + \nabla_a(X^a\mathcal{L}_K\Gamma) + \nabla_a(\Gamma\mathcal{L}_KX^a).
\end{equation}
Lie deriving the trace of the quasi-Einstein equation \eqref{QEE} along $K$ shows that $\nabla_a(\mathcal{L}_KX^a) = X_a\mathcal{L}_KX^a$. Combining with (\ref{lie}) we obtain
\begin{equation} \label{liecomb}
    0 = \Delta (\mathcal{L}_K\Gamma) + X^a\nabla_a(\mathcal{L}_K\Gamma) + (\mathcal{L}_K\Gamma) \nabla_a X^a + K_a\mathcal{L}_K X^a.  
\end{equation}
We can express $\mathcal{L}_K X^a$ in terms of $\mathcal{L}_K \Gamma$ using $X^a = \Gamma^{-1}(K^a - \nabla^a \Gamma)$. Explicitly,
\begin{equation} \label{lkx}
    \mathcal{L}_K X^a = -\Gamma^{-1}(X^a\mathcal{L}_K\Gamma +\nabla^a(\mathcal{L}_K\Gamma)).
\end{equation}
Inserting (\ref{lkx}) into (\ref{liecomb}) and writing $X^a$ in terms of $K^a, \Gamma$ gives the equation $L(\mathcal{L}_K\Gamma) = 0$ where $L$ is the operator defined above.

It remains to prove that ker $L = \{0\}$. By the Fredholm alternative the dimension of ker $L$ is finite and equal to the dimension of ker $L^*$, where $L^*$ is the formal adjoint of $L$ in the $L^2$ inner product. In the case at hand, 
\begin{equation}
    L^*\psi = -\Delta \psi -( \Gamma^{-1}\nabla^a\Gamma)\nabla_a\psi + \vert K \vert^2\Gamma^{-2}\psi.
\end{equation}
This operator is of the form $L^*\psi = -\Delta \psi + B^a\nabla_a \psi + C\psi$ with $C \geq 0$.
Note that $C = \vert K \vert^2\Gamma^{-2}$ is not identically zero because $X^\flat$ is non-gradient. By the strong maximum principle (Theorem 2.9 in \cite{Kazdan1}) it follows that the kernel of $L^*$ is trivial, which concludes the proof.
\end{proof}
\subsection*{Conflict of interest statement}
On behalf of all authors, the corresponding author states that there is no conflict of interest.


\begin{thebibliography}{99}

\bibitem{Bahuaud:2022iao}
E.~Bahuaud, S.~Gunasekaran, H.~K.~Kunduri and E.~Woolgar.
Static near-horizon geometries and rigidity of quasi-Einstein manifolds,
Lett. Math. Phys. \textbf{112} (2022) no.6, 116


\bibitem{BGKW} 
E. Bahuaud, S. Gunasekaran, H. K. Kunduri, and E. Woolgar.
Rigidity of quasi-Einstein metrics: the incompressible case. Lett. Math. Phys. {\bf 114} (2023).

\bibitem{gradientref1} M Brozos-Vazques,. and D. Mojon-Alvarez (2024) Rigidity of weighted Einstein smooth metric measure spaces,
Journal de Math\'ematiques Pures et Appliqu\'ees {\bf 181} (2024),  91-112.



\bibitem{BDE} R. L. Bryant, M. Dunajski, and M. G. Eastwood. Metrisability of two-dimensional projective structures, J. Differential Geometry {\bf 83} (2009), 465--499.

\bibitem{case}
J. Case, Y-J. Shu, G. Wei.
Rigidity of quasi-Einstein metrics.
Diff. Geom. Appl.  {\bf 29} (2011), 93--100.


\bibitem{CRT} P. T. Chru\'sciel, H. S. Reall and P. Tod.
On non-existence of static vacuum black holes with degenerate components of the event horizon.
Class. Quant. Grav. \textbf{23} (2006) 549-554.

\bibitem{alex} A. Colling,  M. Dunajski. Quasi-Einstein structures and Hitchin's equations. 
{\tt arXiv:2504.18475}. (2025)


\bibitem{derdzinski} A. Derdzinski.  Connections with skew-symmetric Ricci tensor on surfaces. 
Result. Math. {\bf 52}, (2008) 223–245.

\bibitem{DK} D. M. Deturck and J. L. Kazdan
 Some regularity theorems in riemannian geometry. Ann. Sci.  Ec. Norm. Sup. {\bf 14},  (1981) ,249-260.


\bibitem{Lew_genus}
D.~Dobkowski-Ry\l{}ko, W.~Kami\'nski, J.~Lewandowski and A.~Szereszewski,
The Near Horizon Geometry equation on compact 2-manifolds including the general solution for g \ensuremath{>} 0,
Phys. Lett. B \textbf{785} (2018), 381-385

\bibitem{D} M. Dunajski. {\em Solitons, Instantons, and Twistors}. 2nd Edition
Oxford Graduate Texts in Mathematics {\bf 31}, OUP (2024)

\bibitem{DL23} M. Dunajski, and J. Lucietti. Intrinsic rigidity of extremal horizons.
J. Differential Geometry. {\bf 132},  (2026) 179-201. {\tt arXiv:2306.17512} 

\bibitem{eastwood} M. Eastwood.
Notes on projective differential geometry.  The IMA Volumes in Mathematics and its Applications 
\textbf{144} (2008)

\bibitem{gradientref2} C. He, P. Petersen,  and W. Wylie  On the classification of warped product
Einstein metrics. Comm.  Analysis. Geom. {\bf 20}, 271–311 (2012).

\bibitem{krynski} W. Kry\'nski.
Webs and projective structures on a plane. 
Diff. Geom. Appl. {\bf 37} (2014) 133.

\bibitem{Kazdan1}  J. L. Kazdan.
Applications of Partial differential Equations to Problems in Geometry. (1993)

\bibitem{KL13}
H.~K.~Kunduri and J.~Lucietti,
Classification of near-horizon geometries of extremal black holes.
Living Rev. Rel. \textbf{16} (2013), 8

\bibitem{KL9}
H.~K.~Kunduri and J.~Lucietti,
A Classification of near-horizon geometries of extremal vacuum black holes.
J. Math. Phys. \textbf{50} (2009), 082502

\bibitem{LP1} J. Lewandowski and T. Pawlowski.  Extremal Isolated Horizons: A
Local Uniqueness Theorem.
Class. Quant. Grav. \textbf{20} (2003) 587-606

\bibitem{LP2} J. Lewandowski, T. Pawlowski. 
Quasi-local rotating black holes in higher dimension: geometry.
Class. Quant .Grav.  \textbf{22} (2005) 1573-1598

\bibitem{milnor}
J. Milnor.  On the existence of a connection with curvature zero. Comment. Math.
Helv. {\bf 32} (1958) 215–223


\bibitem{nur} P. Nurowski and M. Randall.
Generalized Ricci solitons. 
J. Geom. Anal. {\bf 26}  (2016) 1280–1345. 

\bibitem{PT} H. Pedersen and K.P. Tod.
{Three Dimensional Einstein-Weyl Geometry.}
Adv. Math. {\bf 97} (1993) 74-109.

\bibitem{randal} M. Randall.
{Local obstructions to projective surfaces admitting skew-symmetric Ricci tensor}
Jour. Geom. Phys. {\bf 76} (2014) 192.



\end{thebibliography}
\end{document}